\documentclass[oneside,english]{amsart}
\usepackage[T1]{fontenc}
\usepackage[utf8]{inputenc}
\usepackage{array}
\usepackage{multirow}
\usepackage{amstext}
\usepackage{amsthm}
\usepackage{amsmath}
\usepackage{amssymb}
\usepackage{hyperref}

\usepackage[table]{xcolor}

\usepackage[
backend=biber,
style=alphabetic,
sorting=ynt
]{biblatex}
\makeatletter

\numberwithin{equation}{section}
\numberwithin{figure}{section}
\theoremstyle{plain}
\newtheorem{thm}{\protect\theoremname}
  \theoremstyle{definition}
  \newtheorem{defn}[thm]{\protect\definitionname}
  \theoremstyle{remark}
  
  \theoremstyle{plain}
  
  \newtheorem{lem}[thm]{\protect\lemmaname}
  \theoremstyle{definition}
  
  \theoremstyle{plain}
  \newtheorem{cor}[thm]{\protect\corollaryname}
  \providecommand{\corollaryname}{Corollary}

\makeatother

\usepackage{babel}
  \providecommand{\conjecturename}{Conjecture}
  \providecommand{\definitionname}{Definition}
  \providecommand{\examplename}{Example}
  \providecommand{\remarkname}{Remark}
\providecommand{\lemmaname}{Lemma}
\providecommand{\theoremname}{Theorem}

\begin{document}

\title{The Prime Number Theorem and Primorial Numbers}

\author{Jonatan Gomez \\ jgomezpe@unal.edu.co \\ Universidad Nacional de Colombia}
\maketitle
\begin{abstract}
    Counting the number of prime numbers up to a certain natural number and describing the asymptotic behavior of such a counting function has been studied by famous mathematicians like Gauss, Legendre, Dirichlet, and Euler. The prime number theorem determines that such asymptotic behavior is similar to the asymptotic behavior of the number divided by its natural logarithm. In this paper, we take advantage of a multiplicative representation of a number and the properties of the logarithm function to express the prime number theorem in terms of primorial numbers, and $n$-primorial totative numbers. A primorial number is the multiplication of the first $n$ prime numbers while $n$-primorial totatives are the numbers that are coprime to the $n$-th primorial number. By doing this we can define several different functions that can be used to approximate the behavior of the prime counting function asymptotically. 
\end{abstract}

\section{Introduction}
\subsection{Prime Numbers}
A prime number is a natural number having exactly two factors, $1$ and itself \cite{narkiewicz2000}. The natural number $1$ is not considered a prime number since it has only one factor ($1$). From now on, $p_{n}$ will denote the $n$\textbf{\emph{th}} prime number, here $n\geq 1$. The first eleven prime numbers are $2,3,5,7,11,13,17,19,23,29,31$. 
\subsection{Prime Numbers Theorem}
\cite{narkiewicz2000}. Let $\pi(x)$ be the prime number-counting function, i.e., the function that determines the number of primes less than or equal to a natural number $x$. The prime number theorem proved independently by Hadamard \cite{Hadamard1896} and de la Vall\'ee Poussin \cite{delaVallee1896}, describes the asymptotic behavior of $\pi(x)$ as follows: 

\begin{equation}
    \lim_{x \to \infty}\frac{\pi(x)}{\left(\frac{x}{\log(x)}\right)}=1
\end{equation}

Using asymptotic notation, we can rewrite the prime number theorem as follows:

\begin{equation}
\pi(x)\sim\frac{x}{\log(x)}
\end{equation}

Table \ref{tab:pnt} shows the asymptotic behavior of $\pi(x)$ and $\frac{x}{\log(x)}$. Results for $\pi(x)$ are taken from \href{https://en.wikipedia.org/wiki/Prime_number_theorem}{https://en.wikipedia.org/wiki/Prime\_number\_theorem} and results for $\frac{x}{log(x)}$ are obtained with the Python program (PNTprimorials.py) freely available at professor Jonatan Gomez github repository \cite{GomezPrimesGit}.
\begin{table}[htbp]
\centering
\begin{tabular}{|r|r|r|}
\hline
 \multicolumn{1}{|c|}{ $x$ } & \multicolumn{1}{c|}{ $\pi(x)$ } & \multicolumn{1}{c|}{ $\pi(x)/\frac{x}{\log(x)}$ } \\
\hline
$10^{1}$ & 4 & 0.921 \\
$10^{2}$ & 25 & 1.151 \\
$10^{3}$ & 168 & 1.161 \\
$10^{4}$ & 1229 & 1.132 \\
$10^{5}$ & 9592 & 1.104 \\
$10^{6}$ & 78498 & 1.084 \\
$10^{7}$ & 664579 & 1.071 \\
$10^{8}$ & 5761455 & 1.061 \\
$10^{9}$ & 50847534 & 1.054 \\
$10^{10}$ & 455052511 & 1.048 \\
$10^{11}$ & 4118054813 & 1.043 \\
$10^{12}$ & 37607912018 & 1.039 \\
$10^{13}$ & 346065536839 & 1.036 \\
$10^{14}$ & 3204941750802 & 1.033 \\
$10^{15}$ & 29844570422669 & 1.031 \\
$10^{16}$ & 279238341033925 & 1.029 \\
$10^{17}$ & 2623557157654233 & 1.027 \\
$10^{18}$ & 24739954287740860 & 1.025 \\
$10^{19}$ & 234047667276344607 & 1.024 \\
$10^{20}$ & 2220819602560918840 & 1.023 \\
$10^{21}$ & 21127269486018731928 & 1.022 \\
$10^{22}$ & 201467286689315906290 & 1.021 \\
$10^{23}$ & 1925320391606803968923 & 1.020 \\
$10^{24}$ & 18435599767349200867866 & 1.019 \\
$10^{25}$ & 176846309399143769411680 & 1.018 \\
\hline
\end{tabular}
\label{tab:pnt}
\caption{Asymptotic behavior of $\pi(x)$ and $\frac{x}{\log(x)}$. Results for $\frac{x}{log(x)}$ are obtained with the Python program (PNTprimorials.py) freely available at professor Jonatan Gomez github repository \cite{GomezPrimesGit}. Results for $\pi(x)$ are taken from \href{https://en.wikipedia.org/wiki/Prime_number_theorem}{https://en.wikipedia.org/wiki/Prime\_number\_theorem} and }
\end{table}

Now,  we write down some Theorems about prime numbers that have been proven previously in the literature, and we derive some technical properties from such Theorems.

\begin{thm}
    \label{thm:bertrand}
    (\textbf{Bertrand's postulate}) For any positive natural number $n$ we have that $p_{n+1}\leq 2p_n-1$.
\end{thm}
\begin{proof}
    Chebyshev proved this Theorem in \cite{Tchebychev1852}.
\end{proof}
\begin{cor}
    \label{cor:bertrand}
    For any positive natural number $n$ we have that $\log\left(\frac{p_{n+2}}{p_n}\right)<2$
\end{cor}
\begin{proof}
    By applying Theorem \ref{thm:bertrand} twice we have that $p_{n+2}<2p_{n+1}<4p_n$ for all positive natural number $n$. Clearly, $\frac{p_{n+2}}{p_n}<4$ so, $\log\left(\frac{p_{n+2}}{p_n}\right)<\log(4)<2$.
\end{proof}

\begin{thm}
    \label{thm:suzuki}
    For any positive natural number $m$ there exists a positive natural number $N$ such that $p^m_{n+1}<\prod_{i=1}^{n} p_i$ for all $n\geq N$.
\end{thm}
\begin{proof}
    Suzuki proved this Theorem in \cite{Suzuki1913}.
\end{proof}

\begin{cor}
    \label{cor:suzuki}    
    For any positive natural number $m$ there exists a positive natural number $N$ such that $\frac{\log(p_{n+1})}{\log\left(\prod_{i=1}^{n} p_i\right)}<\frac{1}{m}$ for all $n\geq N$.
\end{cor}
\begin{proof}
    Consider inequality in Theorem \ref{cor:suzuki}, apply $\log$ function on both sides of the inequality, divide both sides by the right side, apply property $\log(x^m)=m*\log(x)$ on the left side, and divide both sides of the inequality by $m$.
\end{proof}

\begin{thm}
\label{thm:pdivp_1}
    Functions $f(x)=\prod_{p\leq x}\frac{p}{p-1}$ and $h(x)=e^\gamma * \log(x)$ have the same asymptotic behavior. Here $\gamma$ is the Euler's constant. 
\end{thm}
\begin{proof}
This Theorem can be proven as a Corollary of the Theorem proved by Mertens in \cite{Mertens1874}. 
\end{proof}
\begin{cor}
 \label{cor:pdivp_1}
   Functions $f^{\bullet}(x)=\prod_{p\leq x}p$ and $h^{\bullet}(x)=e^\gamma * \log(x)*\prod_{p\leq x}p-1$ have the same asymptotic behavior. Here $\gamma$ is the Euler's constant. 
\end{cor}

\subsection{\label{Primorial}Primorial}

Many interesting operations can be defined over prime numbers, for example, we can define the \textbf{primorial} of the $n$\emph{th} prime number as the product of the first $n\in\mathbb{N}^{+}$ prime numbers \cite{Dubner1987}, i.e., $\#(n)=p_n\#=\prod_{i=1}^{n}p_k$. This definition is similar to the definition of the factorial function, so it is possible to define the primorial as a recursive function $\#(0)=1$ and $\#(n)=p_n\#(n-1)$ for $n\geq 1$. The first five primorials are $2,6,30,210,2310$.

\subsection{Coprime or Relative-prime Numbers}

Two natural numbers $a, b \in \mathbb{N}$ are coprime or relative-prime numbers iff their greatest common divisor is $1$ ($gcd(a,b)=1$). According to this definition, \textit{i)} $1$ is relative-prime with any other positive natural number, \textit{ii)} $0$ is not relative-prime with any natural number, and \textit{iii)} any prime number is relative-prime with any other prime number. Notice that we can define prime numbers in terms of coprime numbers: A natural number $p>1$ is a prime number iff $p$ is relative-prime to $q$ for all natural numbers $1\leq q<p$.

\subsection{Natural numbers less than $n$ ($\mathbb{Z}_n$)}
Let $n\in\mathbb{N}$, the set of natural numbers less than $n$, is $\mathbb{Z}_n=\{0,1,\ldots,n-1\}$. For example, $\mathbb{Z}_2=\{0,1\}$ and $\mathbb{Z}_{\#(2)}=\mathbb{Z}_{2*3}=\mathbb{Z}_{6}=\{0,1,3,4,5\}$. Amount the properties hold by $\mathbb{Z}_n$, we are especially interested in the structure of the subset of natural numbers that are relative-prime to $\#(n)$ (sometimes called as totative numbers of $\#(n)$). For instance, consider $\mathbb{Z}_{\#(2)}=\mathbb{Z}_6$, the subset of $\mathbb{Z}_6$ defined by the natural numbers that are relative-prime to $6$ is $\{1,5\}$. It is clear that any prime number $q$ such $p_n<q<\#(n)$ will be a relative-prime number to $\#(n)$. 

\subsection{Euler's totient function ($\varphi(n)$)}
Euler's totient function \cite{Euler1763} counts the natural numbers, in $\mathbb{Z}_n$, which are relative-prime to $n$. Take for example $\mathbb{Z}_{\#(3)}=\mathbb{Z}_{30}$, the subset of natural numbers relative-prime to $30$ is $\{1,7,11,13,17,19,23,29\}$, therefore $\varphi(30)=8$. Euler's totient function is a multiplicative function, i.e., $\varphi(ab)=\varphi(a)\varphi(b)$ for two coprime numbers $a$ and $b$. Moreover, for a prime number $p$, we have that $\mathbb{Z}_p - \{0\}$ is the set of totative numbers of $p$, therefore $\varphi(p)= p-1$. Using these two properties of Euler's totient function, we can easily compute it on primorial numbers: $\varphi(\#(n))=\prod_{i=1}^{n}(p_k-1)=\prod_{i=1}^{n}\varphi(p_k)$.

\subsection{Primorial sets}
We define a $n$-primorial set by 'replacing' numbers $0$ and $1$ in $\mathbb{Z}_{\#(n)}$, with numbers $\#(n)$ and $\#(n)+1$, respectively \cite{Gomez2023Primorial}. Let $n\in\mathbb{N}$, the $n$-primorial set is defined as  $\mathbb{Z}^{\#}_n=\{2,3,\ldots,\#(n),\#(n)+1\}$. The following is the list of the first five $n$-primorial sets:

\begin{enumerate}
    \item $\mathbb{Z}^{\#}_1=\{2,3\}$
    \item $\mathbb{Z}^{\#}_2=\{2,3,4,5,6,7\}$
    \item $\mathbb{Z}^{\#}_3=\{2,3,\ldots,30,31\}$
    \item $\mathbb{Z}^{\#}_4=\{2,3,\ldots,210,211\}$
    \item $\mathbb{Z}^{\#}_5=\{2,3,\ldots,2310,2311\}$    
\end{enumerate}

\subsection{\label{sec:tot}Primorial totatives}

We can extend the notion of totatives of $\#(n)$ to set $\mathbb{Z}^{\#}_n$ by considering the subset of natural numbers in $\mathbb{Z}^{\#}_n$ being relative-prime to $\#(n)$. The following is the list of the first three $n$-totative sets: 

\begin{enumerate}
    \item The $1$-totative set is $tot(1)=\{3\}$
    \item The $2$-totative set is $tot(2)=\{5,7\}$
    \item The $3$-totative set is $tot(3)=\{7,11,13,17,19,23,29,31\}$
\end{enumerate}

Notice that any prime number $p \in \mathbb{Z}^{\#}_n$ must be either $p_i$ for some $i=1,2,\ldots,n$ or a $n$-totative number, i.e., we can express the Prime Number Theorem in terms of totative numbers. We explore this relationship in Section \ref{sec:PNTPrimTot}. 
\begin{lem}
\label{lem:tot-phi}
The number of $n$-totatives is $tot(n)=\prod_{i=1}^{n}(p_k-1)$.
\end{lem}
\begin{proof}
$\#(n)$ and $(\#(n)+1)$ are coprime numbers, therefore the number of $n$-totatives is equal to the totatives of $\#(n)$, i.e, the number of $n$-totatives is $\varphi(\#(n))=\prod_{i=1}^{n}\varphi(p_k)=\prod_{i=1}^{n}(p_k-1)=tot(n)$.
\end{proof}

\section{\label{sec:rep}The $\log$ Function and Multiplicative Number Representations}

We can establish relationships between the prime number theorem and primorial numbers if we represent numbers greater or equal to one in terms of multiplicative fractions of prime or primorial numbers. Before that, we develop a technical Lemma for analyzing asymptotic behavior.

\begin{lem}
    \label{lem:squeeze}
    Let $f$, $g$, and $h$ be real number functions such that $0 < f(x) \leq g(x) \leq h(x)$ for all $x$ greater than certain number $N$. If $f$ and $h$ have the same asymptotic behavior then  $f$, $g$, and $h$ have the same asymptotic behavior.
\end{lem}
\begin{proof}
    Obvious, $1=\frac{f(x)}{f(x)}\leq\frac{g(x)}{f(x)}\leq\frac{h(x)}{f(x)}$, by the Squeeze or sandwich Theorem we have $1=\lim_{x \to \infty} \frac{f(x)}{f(x)} \leq \lim_{x \to \infty} \frac{g(x)}{f(x)} \leq \lim_{x \to \infty} \frac{h(x)}{f(x)}=1$.
\end{proof}

\begin{defn}
\label{def:rep}
    Let $A=\left\{a_i\right\}_{i\in \mathbb{N}^{+}}$ be a monotonic increasing succession such that $a_1 \geq 1$. For any number $x \geq a_1$, any positive natural number $n$, and $0\leq r <1$:
    \begin{enumerate}
        \item $n(x)$ as the positive natural number such that $a_{n(x)}\leq x < a_{n(x)+1}$. 
        \item $r(x)=\frac{x-a_{n(x)}}{a_{n(x)+1}-a_{n(x)}}$,
        \item $s(n,r)=1+\left(\frac{a_{n+1}}{a_{n}}-1\right)r$,
        \item $s(x)=s(n(x),r(x))$, and
        \item $y(n,r)=a_{n}*s(n,r)$.
    \end{enumerate}    
\end{defn}

Notice that $s(n,r)=1+\left(\frac{a_{n+1}} {a_{n}}-1\right)r=1+\left(\frac{a_{n+1}-a_{n}}{a_{n}}\right)r = \frac{a_{n+1}r+a_{n}(1-r)}{a_{n}}$ is kind of the $r$-th multiplicative fraction of $a_{n+1}$ respect to succession $A$.

\begin{lem}
\label{lem:rep}
    Let $A=\left\{a_i\right\}_{i\in \mathbb{N}^{+}}$ be a monotonic increasing succession such that $a_1 \geq 1$. For any number $x \geq a_1$, any positive natural number $n$, and $0\leq r <1$
    \begin{enumerate}
        \item $a_{n}\leq y(n,r) < a_{n+1}$,
        \item $x=y(n(x),r(x))=a_{n(x)}*s(x)$
        \item $m=n(y(m,t))$, and $t=r(y(m,t))$
    \end{enumerate}    
\end{lem}        
\begin{proof}
 (1) $y(n,r)=a_{n}*s(n,r)=a_{n}\left(\frac{a_{n+1}r+a_{n}(1-r)}{a_{n}}\right)=a_{n+1}r+a_{n}(1-r)$. Clearly, $y(n,r)$ is a convex combination between $a_{n+1}$ and $a_{n}$, then $a_{n}\leq y(n,r) < a_{n+1}$. (2) $a_{n(x)}*s(x)=a_{n(x)}\left(1+\left(\frac{a_{n(x)+1}-a_{n(x)}}{a_{n(x)}}\right)\left(\frac{x-a_{n(x)}}{a_{n(x)+1}-a_{n(x)}}\right)\right)=a_{n(x)}*s(x)$, i.e., $a_{n(x)}*s(x)=a_{n(x)}\left(1+\frac{x-a_{n(x)}}{a_{n(x)}}\right)=a_{n(x)}\left(\frac{x}{a_{n(x)}}\right)=x$. (3) Follows from (1) and (2).
\end{proof}

We can define, in particular, two different representations of a number $x\geq 2$ if we consider Lemma \ref{lem:rep}:

\begin{enumerate}
    \item The \textbf{prime representation} of $x \geq 2$ uses the succession of prime numbers ($A=\left\{p_n\right\}_{n\in \mathbb{N}^{+}}$). Therefore, $x=p_{n^{\star}(x)}*s^{\star}(x)$ with $n^{\star}(x)$ such that $p_{n^{\star}(x)}\leq x<p_{n^{\star}(x)+1}$, $s^{\star}(x)=1+\left(\frac{p_{n^{\star}(x)+1}}{p_{n^{\star}(x)}}-1\right)r^{\star}(x)$, and $r^{\star}(x)=\frac{x-p_{n^{\star}(x)}}{p_{n^{\star}(x)+1}-p_{n^{\star}(x)}}$.
    \item The \textbf{primorial representation} of $x \geq 2$ uses the succession of primorial numbers ($A=\left\{\#(n)\right\}_{n\in \mathbb{N}^{+}}$). Therefore, $x=\#(n'(x))*s'(x)$ with $n'(x)$ such that $\#(n'(x))\leq x<\#(n'(x)+1)$, $s'(x)=1+\left(p_{n'(x)+1}-1\right)r'(x)$, and $r'(x)=\frac{x-\#(n'(x))}{\#(n'(x)+1)-\#(n'(x))}$.
\end{enumerate}   

We can take advantage of the prime representation of a number for computing the $\log(x)$ function, see Equation \ref{eq:log_prime}.
\begin{equation} \label{eq:log_prime}
\log(x)=\log\left(p_{n^{\star}(x)}*s^{\star}(x)\right)=\log(p_{n^{\star}(x)}) + \log(s^{\star}(x))
\end{equation}

We can manipulate Equation \ref{eq:log_prime} to produce functions having the asymptotic behavior of $\log(x)$, see Equations \ref{eq:logp-} and \ref{eq:logp+}.

\begin{equation} \label{eq:logp-}
\log^{-}(x)=\log\left(p_{n^{\star}(x)-1}\right)
\end{equation}

\begin{equation} \label{eq:logp+}
\log^{+}(x)=\log\left(p_{n^{\star}(x)+1}\right)
\end{equation}

\begin{thm}
\label{thm:logp+-}
Functions $\log^{-}(x)$ and $\log^{+}(x)$ have the same asymptotic behavior, i.e.,
$$
    \lim_{x \to \infty}\frac{\log^{-}(x)}{\log^{+}(x)}=1
$$
\end{thm}
\begin{proof}
Given a real number $\epsilon>0$, we can take $n$ as the smallest natural number such that $\frac{2}{\log(p_n)}<\epsilon$. Now, we consider $M=p_{n}$. If $x>M$ then $n^{\star}(x) > n$. We have $\log^{+}(x)-\log^{-}(x) = \log(p_{n^{\star}(x)+1})-\log(p_{n^{\bullet}(x)-1}) = \log\left(\frac{p_{n^{\star}(x)+1}}{p_{n^{\star}(x)-1}}\right)$. By Corollary \ref{cor:bertrand} we have $\log^{+}(x)-\log^{-}(x) < 2$. Clearly, $\left| \frac{\log^{+}(x)}{\log^{-}(x)} - 1 \right| =\frac{\log^{+}(x)-\log^{-}(x)}{\log^{-}(x)} < \frac{2}{\log^{-}(x)}<\epsilon$. 
\end{proof}
\begin{cor}
    \label{cor:logp+-}
   Functions: 
   \begin{enumerate}
       \item $\log^{-}(x)$
       \item $\log^{a}(x)=\log^{-}(x)+log(a(x))$ with $1 \leq a(x) \leq \frac{p_{n^{\star}(x)+1}}{p_{n^{\star}(x)-1}}$
       \item $\log(x)$
       \item $\log^{\star}(x)=\log(p_{n^{\star}(x)})$
       \item $\log^{*}(x)=\log^{-}(x) + \log\left(1+\left(\frac{p_{n^{\star}(x)}}{p_{n^{\star}(x)-1}}-1\right)r^{\star}(x)\right)$, and
       \item $\log^{+}(x)$
   \end{enumerate}
   Have the same asymptotic behavior.
\end{cor}
\begin{proof}
        Since every function $\log^{?}(x)$ defined in (2)-(6) satisfies $\log^{-}(x)\leq\log^{?}(x)\leq\log^{+}(x)$, then by Theorem \ref{thm:logp+-} and Lemma \ref{lem:squeeze} the proof is completed.
\end{proof}

Now, we take advantage of the primorial representation of a number for computing the $\log(x)$ function, see Equation \ref{eq:log}.

\begin{equation} \label{eq:log}
\log(x)=\log\left(\#(n'(x))*s'(x)\right)=\log(s'(x)) + \log\left(\prod_{i=1}^{n'(x)} p_i\right) = \log(s'(x)) + \sum_{i=1}^{n'(x)} \log(p_i) 
\end{equation}

We can manipulate Equation \ref{eq:log} to produce functions with the asymptotic behavior of $\log(x)$, see Equations \ref{eq:log-} and \ref{eq:log+}.

\begin{equation} \label{eq:log-}
\log_{-}(x)=\log\left(\#(n'(x)-1)\right)=\log\left(\prod_{i=1}^{n'(x)-1} p_i\right)
\end{equation}

\begin{equation} \label{eq:log+}
\log_{+}(x)=\log\left(\#(n'(x)+1)\right)=\log\left(\prod_{i=1}^{n'(x)+1} p_i\right)
\end{equation}

\begin{thm}
\label{thm:log+-}
Functions $\log_{-}(x)$ and $\log_{+}(x)$ have the same asymptotic behavior, i.e.,
$$
    \lim_{x \to \infty}\frac{\log_{-}(x)}{\log_{+}(x)}=1
$$
\end{thm}
\begin{proof}
Given a real number $\epsilon>0$, we can take $m$ as the smallest natural number such that $\frac{2}{m}<\epsilon$. Now, by Theorem \ref{thm:suzuki} we can take $N$ as the smallest natural number such that $p_{n+1}^{m}<\#(n)$ for all $n\geq N$, and by Corollary \ref{cor:suzuki}, we have $\frac{\log(p_{n+1})}{\log(\#(n))}<\frac{1}{m}$ for all $n\geq N$ ($\star$). If $x>M=\#(N)$ then $n'(x)\geq N$ and $\frac{\log(p_{n'(x)+1})}{\log(\#(n'(x)))}<\frac{1}{m}$. Notice that $\log(\#(n'(x)+1))-\log(\#(n'(x)-1))=\log\left(\frac{\#(n'(x)+1)}{\#(n'(x)-1)}\right)=\log(p_{n'(x)}*p_{n'(x)+1})<\log(p^2_{n'(x)+1})$. Clearly, $\left| \frac{\log_{-}(x)}{\log_{+}(x)} - 1 \right| =\frac{\log_{+}(x)-\log_{-}(x)}{\log_{+}(x)}<\frac{2\log(p_{n'(x)+1})}{\log(\#(n'(x)+1))}<\frac{2\log(p_{n'(x)+1})}{\log(\#(n'(x)))}$. Finally, using result ($\star$) we have $\left| \frac{\log^{*}(x)}{\log(x)} - 1 \right| < \frac{2}{m}<\epsilon$.   
\end{proof}
\begin{cor}
    \label{cor:log+-}
   Functions: 
   \begin{enumerate}
       \item $\log_{-}(x)$
       \item $\log_{a(x)}(x)=\log_{-}(x)+log(a(x))$ with $1 \leq a(x) \leq p_{n'(x)}*p_{n'(x)+1}$
       \item $\log(x)$
       \item $\log_{\#}(x)=\log(\#(n'(x)))$
       \item $\log_{*}(x)=\log_{-}(x) + \log(1+(p_{n'(x)}-1)r'(x))$
       \item $\log_{\diamond}(x)=\log_{-}(x) +  \log(n'(x)+r'(x))$, and
       \item $\log_{+}(x)$
   \end{enumerate}
   Have the same asymptotic behavior.
\end{cor}
\begin{proof}
    Since every function $\log_{?}(x)$ defined in (2)-(6) satisfies $\log_{-}(x)\leq\log_{?}(x)\leq\log_{+}(x)$, then by Theorem \ref{thm:log+-} and Lemma \ref{lem:squeeze} the proof is completed.
\end{proof}

\section{Primorials and $n$-Totative Numbers}
Notice that if we consider $x=p_n$, Theorem \ref{thm:pdivp_1} and Corollary \ref{cor:pdivp_1} suggest a relationship between the primorial number $\#(n)$ and the quantity of $n$-totatives numbers when $n \to \infty$. Upon this suggestion, we propose a relationship between any positive number $x\geq 2$ and $n$-totative numbers by considering the primorial representation of a number and defining a function that approximates the quantity of $n$-totative numbers up to $x$, see Equation \ref{eq:tot*}.

\begin{equation}
\label{eq:tot*}
    tot_{*}(x)=tot(n'(x))*t_{*}(x)=t_{*}(x)*\prod_{i=1}^{n'(x)}p_i-1
\end{equation}

Here, $t_{*}(x)=1+(p_{n'(x)+1}-2)*r'(x)$, with $n'(x)$ and $r'(x)$ as defined in Section \ref{sec:rep}, and $tot(n)$ as defined in Section \ref{sec:tot}. Before establishing the relationship, we need the following technical lemma.

\begin{lem}
    \label{lem:aa_1}
    $1 \leq \frac{1+(a-1)r}{1+(a-2)r} < \frac{a}{a-1}$ for all $a>1$ and all $0\leq r <1$.
\end{lem}
\begin{proof}
Consider $r=0$: clearly, $\frac{1+(a-1)r}{1+(a-2)r}=\frac{1}{1}=1$. Consider $0<r<1$: clearly $a-2<a-1$ then $1+(a-2)r<1+(a-1)r$, so $1<\frac{1+(a-1)r}{1+(a-2)r}$. Now $r<1$, i.e., $-1+r<0$. By adding $a+a^2r-2ar$ on both sides of the inequality and organizing terms we have $a-1+a^2r-2ar+r<a+a^2r-2ar$. Grouping terms we have $(a-1)(1+(a-1)r)<a(1+(a-2)r)$ and dividing both sides of inequality by $(a-1)(1+(a-2)r)$ we have $ \frac{1+(a-1)r}{1+(a-2)r} < \frac{a}{a-1}$.
\end{proof}

Now, we follow a similar scheme as the one we used in the previous section to establish the relationship. 

\begin{thm}
\label{thm:pdivp_1}
    Functions $f(x)=\prod^{n'(x)}_{i=1}\frac{p_i}{p_i-1}$ and $g(x)=\prod^{n'(x)+1}_{i=1}\frac{p_i}{p_i-1}$ have the same asymptotic behavior.
\end{thm}
\begin{proof}
    Obvious, $\lim_{x\to\infty}\frac{g(x)}{f(x)}=\lim_{x\to\infty}\frac{p_{n'(x)+1}}{p_{n'(x)+1}-1}=1+\lim_{x\to\infty}\frac{1}{p_{n'(x)+1}-1}=1$
\end{proof}

\begin{cor}
\label{cor:pdivp_1}
    Function $f^{\circ}(x)=\frac{x}{tot_{*}(x)}$ has the same asymptotic behavior of functions $f(x)=\prod^{n'(x)}_{i=1}\frac{p_i}{p_i-1}$ and $g(x)=\prod^{n'(x)+1}_{i=1}\frac{p_i}{p_i-1}$ .
\end{cor}
\begin{proof}
Clearly, $\frac{x}{tot_{*}(x)}=\frac{s'(x)*\#(n'(x))}{t_{*}(x)*\prod_{i=1}^{n'(x)}p_i-1}=f(x)*\frac{s'(x)}{t_{*}(x)}=f(x)*\left(\frac{1+(p_{n'(x)+1}-1)r'(x)}{1+(p_{n'(x)+1}-2)r'(x)}\right)$ by Lemma \ref{lem:aa_1} we have $f(x)\leq\frac{x}{tot_{*}(x)}<f(x)*\frac{p_{n'(x)+1}}{p_{n'(x)+1}-1}=g(x)$. Then by Theorem \ref{thm:pdivp_1} and Lemma \ref{lem:squeeze} the proof is completed.
\end{proof}

Finally, we use Theorem \ref{thm:pdivp_1} to establish the desired relationship.

\begin{thm}
\label{thm:pdivp_1-2}
    Functions $f(x)=\prod^{n'(x)}_{i=1}\frac{p_i}{p_i-1}$ and $g^{\circ}(x)=e^\gamma * \log(y(x))$ with $\gamma$ the Euler's constant and $y(x)=p_{n'(x)}*\left(1+\left(\frac{p_{n'(x)+1}}{p_{n'(x)}}-1\right)r'(x)\right)$ have the same asymptotic behavior. 
\end{thm}
\begin{proof} 
    Given a real number $\epsilon>0$, we consider Theorem \ref{thm:pdivp_1} and take $M$ such that $\left|\frac{\prod_{p\leq z}\frac{p}{p-1}}{e^\gamma * \log(z)}-1\right|<\epsilon$ for all $z>y(M)$. If $x>M$ then $y(x) > y(M)$ and $\left|\frac{\prod_{p\leq y(x)}\frac{p}{p-1}}{e^\gamma * \log(y(x))}-1\right|<\epsilon$. We can see that $\prod^{n'(x)}_{i=1}\frac{p_i}{p_i-1}=\prod_{p\leq y(x)}\frac{p}{p-1}$ therefore, $\left|\frac{\prod^{n'(x)}_{i=1}\frac{p_i}{p_i-1}}{e^\gamma * \log(y(x))}-1\right|<\epsilon$.
\end{proof}

\begin{cor}
\label{cor:pdivp_1-2}
Functions $f^{\circ}(x)=\frac{x}{tot_{*}(x)}$ and $g^{\circ}(x)=e^\gamma * \log(y(x))$ with $\gamma$ the Euler's constant and $y(x)=p_{n'(x)}*\left(1+\left(\frac{p_{n'(x)+1}}{p_{n'(x)}}-1\right)r'(x)\right)$ have the same asymptotic behavior.
\end{cor}
\begin{proof}
Follows from Corollary \ref{cor:pdivp_1} and Theorem \ref{thm:pdivp_1-2}.    
\end{proof}

\begin{cor}
\label{cor:pdivp_1-3}
Functions $id(x)=x$ and $h^{\circ}(x)=tot_{*}(x)*e^\gamma * \log(y(x))$ with $\gamma$ the Euler's constant and $y(x)=p_{n'(x)}*\left(1+\left(\frac{p_{n'(x)+1}}{p_{n'(x)}}-1\right)r'(x)\right)$ have the same asymptotic behavior.
\end{cor}
\begin{proof}
Follows from Corollary \ref{cor:pdivp_1-2}.    
\end{proof}

\section{\label{sec:PNTPrimTot}Putting All Together}
Notice that functions having the same asymptotic behavior as functions $id(x)$ and $\log(x)$ can replace them in the prime number theorem. So, we can express the Prime Number Theorem in several ways, one of them using $n$-totative numbers. Table \ref{tab:pntprimorial} shows the asymptotic behavior of $\pi(x)$ vs $\frac{x}{\log(x)}$, $\frac{x}{\log_{*}(x)}$, $\frac{x}{\log_{\diamond}(x)}$, $\frac{h^{\circ}(x)}{\log(h^{\circ}(x))}$, and $\frac{x}{\log(h^{\circ}(x))}$. These results are obtained with the Python program (PNTprimorials.py) freely available at Professor Jonatan Gomez github repository \cite{GomezPrimesGit}.

\begin{table}[htbp]
\centering
\begin{tabular}{|r|r|r|r|r|r|}
\cline{2-6}
\multicolumn{1}{c}{} & \multicolumn{5}{|c|}{$\pi(x)/$ } \\
\hline
 \multicolumn{1}{|c|}{$x$} & \multicolumn{1}{c|}{$\frac{x}{\log(x)}$} & \multicolumn{1}{c|}{$\frac{x}{\log^{*}(x)}$} & \multicolumn{1}{c|}{$\frac{x}{\log^{\diamond}(x)}$} & \multicolumn{1}{c|}{$\frac{h^{\circ}(x)}{\log(h^{\circ}(x))}$} & \multicolumn{1}{c|}{$\frac{x}{\log(h^{\circ}(x))}$} \\
\hline
$10^{1}$ & 0.921 & 0.392 & 0.702 & 1.101 & 0.787 \\
$10^{2}$ & 1.151 & 0.683 & 0.988 & 1.327 & 1.106 \\
$10^{3}$ & 1.161 & 0.770 & 1.018 & 1.305 & 1.137 \\
$10^{4}$ & 1.132 & 0.820 & 1.025 & 1.236 & 1.119 \\
$10^{5}$ & 1.104 & 0.840 & 1.014 & 1.227 & 1.093 \\
$10^{6}$ & 1.084 & 0.858 & 1.011 & 1.180 & 1.077 \\
$10^{7}$ & 1.071 & 0.875 & 1.013 & 1.186 & 1.064 \\
$10^{8}$ & 1.061 & 0.881 & 1.004 & 1.145 & 1.057 \\
$10^{9}$ & 1.054 & 0.885 & 0.997 & 1.137 & 1.050 \\
$10^{10}$ & 1.048 & 0.893 & 0.998 & 1.103 & 1.045 \\
$10^{11}$ & 1.043 & 0.902 & 0.998 & 1.089 & 1.041 \\
$10^{12}$ & 1.039 & 0.905 & 0.995 & 1.102 & 1.037 \\
$10^{13}$ & 1.036 & 0.910 & 0.996 & 1.081 & 1.034 \\
$10^{14}$ & 1.033 & 0.914 & 0.995 & 1.069 & 1.032 \\
$10^{15}$ & 1.031 & 0.919 & 0.996 & 1.072 & 1.030 \\
$10^{16}$ & 1.029 & 0.924 & 0.997 & 1.061 & 1.028 \\
$10^{17}$ & 1.027 & 0.926 & 0.996 & 1.077 & 1.026 \\
$10^{18}$ & 1.025 & 0.929 & 0.996 & 1.076 & 1.024 \\
$10^{19}$ & 1.024 & 0.931 & 0.995 & 1.065 & 1.023 \\
$10^{20}$ & 1.023 & 0.933 & 0.995 & 1.061 & 1.022 \\
$10^{21}$ & 1.022 & 0.935 & 0.995 & 1.046 & 1.021 \\
$10^{22}$ & 1.021 & 0.938 & 0.995 & 1.049 & 1.020 \\
$10^{23}$ & 1.020 & 0.940 & 0.996 & 1.042 & 1.019 \\
$10^{24}$ & 1.019 & 0.941 & 0.995 & 1.054 & 1.018 \\
$10^{25}$ & 1.018 & 0.943 & 0.996 & 1.048 & 1.017 \\
\hline
\end{tabular}
\label{tab:pntprimorial}
\caption{Asymptotic behavior of $\pi(x)$, $\frac{x}{\log(x)}$, $\frac{x}{\log_{*}(x)}$, $\frac{x}{\log_{\diamond}(x)}$, $\frac{h^{\circ}(x)}{\log(h^{\circ}(x))}$, and $\frac{x}{\log(h^{\circ}(x))}$. Results are obtained with the Python program (PNTprimorials.py) freely available at Professor Jonatan Gomez github repository \cite{GomezPrimesGit}.}
\end{table}

\section{Conclusions and Future Work}
We have developed several functions having the same asymptotic behavior of $\pi(x)$ the prime number counting function. We do this by representing any positive positive number in terms of primorial numbers and multiplicative fractions of prime numbers. We were also able to define a function with the same asymptotic behavior of $\pi(x)$ but defined in terms of primorial $n$-totative numbers. We study this relationship since any prime number lower or equal than $\#(n)+1$ is a prime number $p_i$ with $i=1,2,\ldots,n$ or a primorial $n$-totative number, i.e., $\pi(x)$ can be expressed in terms of $tot(n)$.

Our future work will concentrate on defining the class of functions with the same asymptotic behavior of $\pi(x)$, defining a function with a smaller convergence ratio to $\pi(x)$ for small values of $x$. We will in-depth study such functions but defined in terms of primorial $n$-totative numbers and define some function approximations for counting twin, cousin, sexy primes, and a constellation of prime numbers in the same way we did for $\pi(x)$ in terms of $n$-totative numbers.

\printbibliography
\end{document}